\documentclass[11pt]{article}

\usepackage{fullpage}
\usepackage{amsthm}
\usepackage{amssymb}
\usepackage{amsmath}
\usepackage{graphicx}
\usepackage{tikz}
\usepackage{mathrsfs}
\usepackage{float}
\usepackage{makecell}
\usepackage{amscd}
\usepackage{multirow}
\usepackage{verbatim}
\usepackage{hhline}

\newcommand{\Z}{\mathbb{Z}}
\newcommand{\N}{\mathbb{N}}

\DeclareMathOperator{\skel}{skel}

\usepackage{relsize}
\newtheorem{theorem}{Theorem}
\newtheorem{lemma}[theorem]{Lemma}
\newtheorem*{unnumtheorem}{Theorem}

\newtheorem*{lovaszlemma}{Lov\'asz Local Lemma}
\newtheorem{proposition}[theorem]{Proposition}

\theoremstyle{definition}
\newtheorem*{conjecture}{Conjecture}
\newtheorem{definition}[theorem]{Definition}

\title{A lower bound on the number of homotopy types of simplicial complexes on $n$ vertices} 
\author{Andrew Newman \thanks{Some of this work was conducted at Ohio State University and funded in part by the National Science Foundation grants NSF-DMS \#1547357 and \#60041693, and the rest was conducted at Technische Universit\"at Berlin funded by Deutsche Forschungsgemeinshaft (DFG, German Research Foundation) Graduiertenkolleg ``Facets of Complexity" (GRK 2434)}}
\date{\today}
\begin{document}
\maketitle
\abstract
For $n \in \N$, let $h(n)$ denote the number of simplicial complexes on $n$ vertices up to homotopy equivalence. Here we prove that $h(n) \geq 2^{2^{0.02n}}$ when $n$ is large enough. Together with the trivial upper bound of $2^{2^n}$ on the number of labeled simplicial complexes on $n$ vertices this proves a conjecture of Kalai that $h(n)$ is doubly exponential in $n$. \\

\section{Introduction}
In \cite{Newman}, the present author defines $T_d(G)$ for $d \geq 2$ and $G$ a finite abelian group to be the smallest integer $n$ so that there exists a $d$-dimensional simplicial complex $X$ on $n$ vertices with the torsion part of $H_{d - 1}(X)$ (denoted $H_{d - 1}(X)_T$) isomorphic to $G$. The main result in \cite{Newman} is the upper bound in the following theorem. The lower bound was previously known due to \cite{Kalai}.
\begin{unnumtheorem}[Theorem 1.1 from \cite{Newman}]For every fixed $d \geq 2$, there exist constants $c_d$ and $C_d$ so that for any finite abelian group $G$, $$c_d(\log |G|)^{1/d} \leq T_d(G) \leq C_d(\log |G|)^{1/d}.$$
\end{unnumtheorem}

While the research direction leading to this result was motivated by experiments conducted in \cite{KLNP} examining the torison burst in random complexes, the result also says something about the number of homotopy types of simplicial complexes on $n$ vertices. Indeed, another formulation of this result is that for every $d \geq 2$, there exists a constant $C_d$ so that for every $n$, \emph{every} abelian group $G$ of size at most $\exp(n^d/C_d^d)$ is realizable as the torsion subgroup of the $(d - 1)$st homology group for some simplicial complex on $n$ vertices. If we let $h(n, d)$ for $n, d \in \N$ denote the number of $d$-dimensional simplicial complexes on $n$ vertices up to homotopy type, the bound on $T_d(G)$ implies that $h(n, d) \geq \exp(n^d/C_d^d)$. Therefore if $h(n)$ denotes the number of homotopy types of simplicial complexes on $n$ vertices without restriction on the dimension, then we have that $\log h(n)$ grows faster than any polynomial in $n$. Thus it becomes natural to ask how much faster.

Of course there is an obvious upper bound on $h(n)$. In particular, $h(n) \leq 2^{2^n}$ since this bounds the number of labeled simplicial complexes on $n$ vertices. An improved upper bound comes from the $n$th Dedekind number $M(n)$ which counts exactly the number of labeled simplicial complexes on $n$ vertices and \cite{KM} shows that $\log M(n) \leq (1 + o(1)) \binom{n}{\lfloor n/2 \rfloor}$.

The question of the number of homotopy types of simplicial complexes on $n$ vertices appears on MathOverflow as a question of Vidit Nanda \cite{Nanda}, where Gil Kalai conjectures \cite{KalaiMO} that the number of homotopy types of simplicial complexes on $n$ vertices is doubly exponential in $n$. Here we modify the argument from \cite{Newman} to give the following lower bound on $h(n)$ to show that this is the case.
\begin{theorem}\label{thm:main}
For $n$ large enough, $h(n) \geq 2^{2^{0.02n}}$.
\end{theorem}

We point out that the approach of examining which torsion groups can appear in homology is better than the best lower bound one could hope for by counting, say the possible sequence of Betti numbers.  Indeed if $X$ is a simplicial complex on $n$ vertices then for all $0 \leq i < n$, $$\beta_i(X) \leq \binom{n}{i + 1} \leq 2^n,$$ and $\beta_i = 0$ for $i \geq n $. Thus the number of possible sequences of Betti numbers for a simplicial complex on $n$ vertices is upper bounded by $2^{n^2}$.

Unfortunatey, we cannot directly apply the main result of \cite{Newman}, as the upper bound constant $C_d$ is too large, directly from the proof of the result itself, it is exponential in $d$. Indeed even going through the argument in \cite{Newman} more carefully still gives a constant $C_d$ that is $\Theta(d^7)$ and here we need a constant $C_d$ which grows linearly in $d$. Such a bound can be obtained in general, however this requires going though the entire argument from \cite{Newman} and improving the constants in several place in ways that are not particularly interesting in general. To more efficiently prove the main result here, we restrict to a special case. Namely, we consider abelian 2-groups rather than general finite abelian groups and prove the following.

\begin{theorem}\label{thm:2torsionTheorem}
For $d$ large enough and $G$ an abelian group of order $2^{2^d}$ there exists a simplicial complex $X$ on $25d$ vertices of dimension $d$ with $H_{d - 1}(X)_T \cong G$.
\end{theorem}
\begin{proof}[Proof of Theorem \ref{thm:main} from Theorem \ref{thm:2torsionTheorem}]
It is well known that for $m \in \N$ the number of nonisomorphic abelian groups of order $m$ is given by $\pi(e_1) \pi(e_2) \cdots \pi(e_k)$ where $m = p_1^{e_1}p_2^{e_2} \cdots p_k^{e_k}$ is the prime factorization of $m$ and $\pi$ is the partition counting function. Therefore, for any $d$, the number of abelian groups of order $2^{2^d}$ is given by $\pi(2^d)$. A result of Hardy and Ramanujan \cite{HardyRamanujan} gives the asymptotic rate of growth for $\pi(\cdot)$, and here gives that $\log_2(\pi(2^d))$ is asymptotic to $C 2^{d/2}$ where $C > 1$ is an explicit constant. Therefore for $d$ large enough there are at least $2^{2^{d/2}}$ abelian groups of order $2^{2^d}$ and each of them may appear a $H_{d - 1}(X)_T$ for some simplicial complex $X$ on $25d$ vertices. Thus for $n$ large enough, there are at least $2^{2^{n/50}}$, $(n/25)$-dimensional complexes on $n$ vertices which each have a unique group of order $2^{2^{n/25}}$ as the torsion part of the codimension 1 homology group. 
\end{proof}

Note that in the proof above and elsewhere below above we implicitly treat certain real numbers as integers; to simplify the notation we suppress floor and ceiling symbols.

\section{Construction of simplicial complexes with prescribed torsion in homology}
Here we briefly review the proof of the main result of \cite{Newman}. We will modify certain parts of this proof in the proof of Theorem 2, but the idea is the same. There are two steps to the proof in \cite{Newman}, for fixed $d$ the two steps are:
\begin{enumerate}
\item Show that for every $d \geq 2$ and every abelian group $G$ there exists a $d$-dimensional simplicial complex $X = X(G)$ on $O_d(\log |G|)$ vertices whose 1-skeleton has maximum degree $O_d(1)$ and $H_{d - 1}(X)_T \cong G$. This is accomplished via an explicit construction.

\item Given $G$ build the initial construction $X$ from Step 1. Using the probabilistic method, in particular the Lov\'{a}sz Local Lemma, find a proper coloring $c$ of the vertices of $X$ having at most $O_d(\sqrt[d]{|V(X)|})$ colors so that no two $(d - 1)$-dimensional faces of $X$ receive the exact same set of colors on their vertices. The ``final construction" denoted $(X, c)$ is a quotient of $X$ determined by the coloring $c$ and proves the upper bound in the theorem. 
\end{enumerate}

The initial construction here will remain the same, though we restrict to the case of abelian 2-groups. In general the initial construction is call a sphere-and-telescope construction. After a simple reduction to the case of cyclic groups, one builds $X$ so that $H_{d - 1}(X)_T = \Z/m\Z$ via a repeated squares presentations. Suppose that $m = 2^{n_0} + \cdots + 2^{n_k}$, then $\Z/m\Z$ is given by the abelian group presentation $\langle a_1, a_2, ..., a_{n_k} \mid 2a_1 = a_2, 2a_2 = a_3, ..., 2a_{n_k - 1} = a_{n_k}, a_{n_0} + a_{n_1} + \cdots + a_{n_k} = 0 \rangle.$ The ``telescope" portion of the construction gives the relators $2a_i = a_{i + 1}$, with the ``sphere" (or more precisely punctured sphere) portion of the construction giving the last relator. Thus for powers of 2, we only need the telescope portion making the initial construction easier to describe.

For the final construction $(X, c)$ is given as a quotient of $X$ by the following definition.
\begin{definition}[Definition 2.2 from \cite{Newman}]
If $X$ is a simplicial complex with a coloring $c$ of $V(X)$ we define the \emph{pattern} of a face to be the multiset of colors on its vertices. If $c$ is a proper coloring, in the sense that no two vertices connected by an edge receive the same color, we define the \emph{pattern complex} $(X, c)$ to be the simplicial complex on the set of colors of $c$ so that a subset $S$ of the colors of $c$ is a face of $(X, c)$ if and only if there is a face of $X$ with $S$ as its pattern.
\end{definition}
Now, if $X$ is $d$-dimensional and $c$ gives a unique pattern to every $(d - 1)$ dimensional face of $X$, then the torsion part of $H_{d - 1}(X)$ is the same as that of $H_{d - 1}((X, c))$. This appears as Lemma 2.3 in \cite{Newman}. This is why the second step, accomplished via the probabilistic method, is to prove that such a coloring exists.

\section{Proof of Theorem \ref{thm:2torsionTheorem}}
For \cite{Newman}, it sufficed to have constants in the statements and proofs that depended on $d$, without actually computing those constants. In the present situation, we need to understand how these constants depend on $d$. Since we have a different focus here we don't worry about obtaining a constant $C_d$ that is linear in $d$ in general, rather opting for the special case of abelian 2-groups. Moreover for 2-groups the deterministic part of the construction is the same as in \cite{Newman}, the probabilistic coloring part still uses the Lov\'{a}sz Local Lemma but in a different way than in \cite{Newman}. Thus we begin with the probabilistic step. 


To state the coloring lemma we reintroduce some notation from \cite{Newman}. If $X$ is a simplicial complex, we denote by $\skel_i(X)$ the set of $i$-dimensional faces of $X$ and for nonnegative integers $i$ and $j$, we denote by $\Delta_{i, j}(X)$ the maximum of degree of an $i$-dimensional face in $j$-dimensional faces. That is 
\[\Delta_{i, j}(X) = \max_{\sigma \in \skel_i(X)} |\{\tau \in \skel_j(X) \mid \sigma \subseteq \tau \}|.\]
\begin{lemma}\label{newcolor}
Let $X$ be a $d$-dimensional simplicial complex, for $d \geq 2$, on $n$ vertices with $\Delta_{0, d-1}(X) \leq L$ for some $L$. If there exists a proper coloring $c$ of the vertices of $X$ having at most $K$ colors so that no pair of intersecting $(d - 1)$-dimensional faces receive the same pattern by $c$ then there exists a second proper coloring $c'$ of $V(X)$ having at most $K(3eL^2n)^{1/d}$ colors so that no pair of $(d-1)$-dimensional faces of $X$ receive the same pattern by $c'$.
\end{lemma}

As in \cite{Newman}, we will make use of the Lov\'{a}sz Local Lemma which we state here as it appears in the symmetric case in \cite{AS}: 
\begin{lovaszlemma}[\cite{EL}]
Let $A_1, A_2, ..., A_n$ be events in an arbitrary probability space. Suppose that each event $A_i$ is mutually independent of all the other events $A_j$ but at most $t$, and that $\Pr[A_i] \leq p$ for all $1 \leq i \leq n$. If $ep(t+1) \leq 1$ then $\Pr[\bigwedge_{i = 1}^n \overline{A_i}] > 0$
\end{lovaszlemma}

\begin{proof}[Proof of Lemma \ref{newcolor}]
Let $X$ be a simplicial complex colored by $c$ as described in the statement. For $c'$, which handles the disjoint $(d-1)$-dimensional faces, we will take the product of $c$ with a second coloring $c_2$ found using the Lov\'{a}sz Local Lemma by considering the random process of coloring every vertex independently from a set of $(3eL^2n)^{1/d}$ colors. Moreover, we will also use the fact that we have colored the vertices properly by $c$. For a vertex-disjoint pair of $(d-1)$-dimensional faces $(\sigma, \tau)$, let $A_{\sigma, \tau}$ be the event that $\sigma$ and $\tau$ receive the same pattern by the coloring $c' = (c, c_2)$. If $\sigma$ and $\tau$ receive different patterns by $c$, then $c_2$ is irrelevant to guaranteeing that they do not receive the same pattern by $c'$. Thus we only have to consider vertex-disjoint pairs $(\sigma, \tau)$ where $\sigma$ and $\tau$ receive the same pattern by $c$. Since $c$ is a proper coloring of $V(X^{(1)})$ we have that the identical colorings of $\sigma$ and $\tau$ induce a bijection $\varphi: V(\sigma) \rightarrow V(\tau)$ by sending each vertex in $\sigma$ to the unique vertex in $\tau$ that received the same coloring under $c$. Thus $\sigma$, $\tau$ receive the same pattern by $c'$ if and only if for every $v \in \sigma$, $c_2(v) = c_2(\varphi(v))$. For each $v$, this occurs with probability $(3eL^2n)^{-1/d}$. Thus for every vertex-disjoint pair $(\sigma, \tau)$ we have
$$\Pr(A_{\sigma, \tau}) \leq \dfrac{1}{((3eL^2n)^{1/d})^d} = \frac{1}{3eL^2n}$$
Now, we need to bound the number of $(\sigma', \tau')$ which have $A_{\sigma, \tau}$ not independent from $A_{\sigma', \tau'}$. This is bounded above by $2dL(Ln/d) = 2L^2 n$ (pick one of the at most $dL$ faces $\sigma'$ which share a vertex with $\sigma$, then pick any of the at-most $Ln/d$ faces for $\tau$ and  multiply by two since we may reverse the role of $\sigma$ and $\tau$). Now, we have 
$$e \frac{1}{3eL^2n} (2L^2 n + 1) \leq 1,$$
so the Lov\'{a}sz Local Lemma applies. Thus there is a coloring $c_2$, so that the resulting product $c'$ of $c$ and $c_2$ has that no two $(d-1)$-dimensional faces receive the same pattern.
\end{proof}

Now we just need to find $K$, $L$, and $n$ from a colored initial construction in order to apply Lemma \ref{newcolor} to prove Theorem \ref{thm:2torsionTheorem}. We will prove the following theorem regarding the initial construction restricting to the case of 2-groups. 

\begin{lemma}\label{initial}
Let $d \geq 2$ be an integer. Then for every abelian 2-group $G$ there exists a $d$-dimensional simplicial complex $X = X(G)$ with $H_{d - 1}(X)_T \cong G$ so that the following bounds hold:
\[|V(X)| \leq 2(d + 1) \log_2 |G| \]
and 
\[\Delta_{0, (d - 1)}(X) \leq 2((2 + o_d(1))^dd)  \]
And moreover there exists a proper coloring of $V(X)$ using at most $3(d + 1)$ colors so that no pair of intersecting $(d - 1)$-dimensional faces receive the same pattern.
\end{lemma}

Before we prove this we will use it to prove Theorem \ref{thm:2torsionTheorem}.
\begin{proof}[Proof of Theorem \ref{thm:2torsionTheorem}]
Fix $d \geq 2$ and $G$ an abelian 2-group. We apply Lemma \ref{newcolor} to the initial construction $X$ for $G$ given by Lemma \ref{initial} with $K = 3(d + 1)$, $L = 2((2 + o_d(1))^dd)$, and $n \leq 2(d + 1) \log_2|G|$. Thus there exists a coloring of $X$ satisfying the conclusion of Lemma \ref{newcolor} having at most $$(3d + 3)(3e(2(2 + o_d(1))^d d)^2)^{1/d}(\log_2|G|)^{1/d}$$ colors. As $d$ tends to infinity, this approaches $3d(4 \log_2^{1/d}|G|) = 12d \log_2^{1/d}|G|$. Thus for $d$ large enough, the final construction for any $G$ with $|G| = 2^{2^d}$ will have at most $25d$ vertices.
\end{proof}

The full proof of Lemma \ref{initial} is omitted here as the construction is identical to the construction for the case of abelian 2-groups in \cite{Newman}. Instead we only describe the construction and its coloring. In doing so we prove every part of Lemma \ref{initial} except that the homology group is correct, but this should be believable from the description and the full details to verify the homology group are given in Section 3 of \cite{Newman}.

The first observation to prove Lemma \ref{initial} is that we may restrict to the case that $G$ is cyclic. Indeed, suppose that $G = \Z/(2^{e_1})\Z \oplus \Z/(2^{e_2})\Z \oplus \cdots \oplus \Z / (2^{e_l})\Z$ and we have Lemma \ref{initial} for cyclic groups. Then for each $e_i$ there exists $X_i$, $d$-dimensional colored by colors $\{1, ..., 3(d + 1)\}$ so that no intersecting pair of $(d - 1)$-dimensional faces receive the same pattern. And for each $i$, $H_{d - 1}(X_i) = \Z/(2^{e_i})\Z$, $\Delta_{0, (d - 1)}(X_i) \leq 2((2 + o_d(1))^dd)$ and $X_i$ has at most $2(d + 1)e_i$ vertices. Now taking $X$ to be the disjoint union of the $X_i$ (with the vertex colorings remaining unchanged) gives us Lemma \ref{initial} for $G$. 

The construction for cyclic 2-groups is now given by first defining a family of building blocks $P(d)$ for $d \geq 2$. Each $P(d)$ is a simplicial complex on $2(d + 1)$ vertices with $H_{d - 1}(P(d)) = \langle a, b \mid 2a = b \rangle$, with the homology class of $a$ being represented by the boundary of a $d$-simplex and the homology class $b$ being represented by the boundary of another $d$-simplex disjoint from the first. Note that the $2(d + 1)$ vertices in these two simplex boundaries constitute all the vertices of $P(d)$. Now for given $d, t \in \N$ with $d \geq 2$, take $t$ copies of $P(d)$, denoted $P_1, ..., P_t$ with $H_{d - 1}(P_i) = \langle a_i, b_i \mid 2a_i = b_i \rangle$. Next for $i \leq t - 1$, glue $P_i$ to $P_{i + 1}$ by identifying the simplex boundary representing $b_i$ to the simplex boundary representing $a_{i+1}$. Finally fill in the $d$-simplex boundary representing $b_t$ with a $d$-simplex and we have a complex $X = X(d, t)$ with $H_{d - 1}(X) = \langle a_1, ..., a_t, b_t \mid 2a_1 = a_2, 2a_2 = a_3, ..., 2a_{t - 1} = a_t, 2a_t = b_t, b_t = 0 \rangle = \langle a_1 \mid 2^t a_1 = 0 \rangle \cong \Z/2^t\Z$.

Now clearly the number of vertices in $X(d, t)$ is $(d + 1)(t + 1) \leq (d + 1) (\log_2 |\Z/2^t\Z| + 1) \leq 2(d + 1) \log_2|\Z/2^t\Z|$. Moreover as each vertex is contained in at most $2$ copies of $P(d)$, we have that $\Delta_{0, {d - 1}}(X(d, t))$ is at most twice the number of $(d - 1)$-dimensional faces of $P(d)$. We should therefore describe how $P(d)$ is constructed.

The complex $P(d)$ is constructed by induction on the dimension. For $P(2)$, we start with the triangulation of the real projective plane with 6 vertices, 15 edges, and 10 faces given by antipodal identification of the icosahedron, and then delete a single face. Explicitly, $P(2)$ can be described as the pure simplicial complex on vertex set $\{1, 2, 3, 4, 5, 6\}$ with 2-dimensional faces 
\[[1, 2, 6], [1, 3, 6], [2, 4, 6], [3, 5, 6], [2, 3, 4], [2, 3, 5], [1, 3, 4], [1, 4, 5], [1, 2, 5].\]

For the induction with $d \geq 3$ we observe that $P(d - 1)$ has $2d$ vertices which define two disjoint $(d - 1)$-simplex boundaries each representing a homology class. Let $v_1, ..., v_d$ and $u_1, ..., u_d$ be the vertices defining these two simplex boundaries. Now construct $P(d)$ by taking the suspension of $P(d - 1)$ with suspension points $v'$ and $u'$, then add two new $d$-dimensional faces $[u', v_1, ..., v_d]$ and $[v', u_1, ..., u_d]$ (as well as the $(d-1)$-dimensional faces $[v_1, ..., v_d]$ and $[u_1, ..., u_d]$). In $P(d)$ the two simplex boundaries representing the relevant homology class are the simplex boundary on $v', v_1, v_2, ..., v_d$ and the simplex boundary on $u', u_1, u_2, ..., u_d$. \footnote{When formally defining $P(d)$ in \cite{Newman}, it is necessary to be careful with how the simplicies are oriented so that things may be properly attached to give the right torsion group. If you want a cyclic group of order $2^0 + 2^3 + 2^6 = 73$, you don't want to attach something ``backwards" and end up with $2^0 - 2^3 + 2^6 = 57$ instead, for instance. We exclude these details here.}

It is clear that $P(d + 1)$ will have $2d + 2$ vertices and will have, if we let $f_i$ denote the usual $i$th entry of the $f$-vector of a complex, $2(f_d(P(d))) + 2$ top-dimensional faces with $P(2)$ having 9 2-dimesional faces. Thus we obtain the claimed bound on the number of $(d - 1)$-dimensional faces. 

The last piece to check is that we may color $X = X(d, t)$ with $3(d + 1)$ colors so that no pair of intersecting $(d - 1)$-dimensional faces receives the same pattern. We have that $X$ is given by $t$ copies of $P_i$ so its vertices come from $t + 1$, $d$-simplex boundaries. Denote these simplex boundaries $S_0, S_1, ..., S_t$ and properly color the vertices of $S_i$ arbitrarily with colors $\{1, ..., d + 1\}$ for $i \equiv 0 \mod 3$, with colors $\{d + 2, ..., 2d + 2\}$ for $i \equiv 1 \mod 3$, and with colors $\{2d + 3, ..., 3d + 3\}$ for $i \equiv 2 \mod 3$. Now any pair of vertices at distance (in the 1-skeleton of $X$) at most 2 from one another receive different colors so there can be no intersecting $(d - 1)$-dimensional faces that receive the same pattern. This now completes the sketch of the argument for Lemma \ref{initial}. With this we have completed the proof of Theorem \ref{thm:main}.

\section{Remarks}
For the proof of the main result, we restricted to the case of constructing complexes with prescribed abelian 2-groups in homology. This was to make the argument simpler as it relates to the main focus here. It is possible to prove that for any $d \geq 2$ and any abelian group $G$, $T_d(G) \leq Cd\sqrt[d]{\log |G|}$ where $C$ is an absolute constant. Chapter 4 of the author's Ph.D. thesis \cite{thesis} gives such an argument with $C = 50$. However for $|G|$ not a power of 2 the deterministic construction is more complicated as is the first round coloring. Such an argument could be worth going though here if it could be pushed further to give that for any $\epsilon > 0$ and $n$ large enough $\log_2 \log_2 h(n) \geq (1 - \epsilon) n$, but this does not seem to be the case.

In the other direction, the connection to enumerating homotopy types does give the following proposition:
\begin{proposition}
Let $C < \frac{\log_2(e)}{e}$, then there exists $d$ arbitrarily large and a finite abelian group $G$ so that $T_d(G) > Cd\sqrt[d]{\log_2|G|}$. 
\end{proposition}
\begin{proof}
Suppose that for every $d$ large enough and every finite abelian group $G$ that $T_d(G) \leq Cd \sqrt[d]{\log_2 |G|}$ with $C < \log_2(e)/e$. Then every cyclic group of order $e^{e^d}$ is realizable as the torsion part of the $(d - 1)$st homology group for a simplicial complex on $Cde\sqrt[d]{\log_2(e)}$ vertices. For $d$ large enough this is at most $C_1 d e$ for some $C_1$ satisfying $C < C_1 < \frac{\log_2(e)}{e}$. Thus setting $n = C_1 d e$ we have at least $e^{e^{n/(C_1e)}}$ homotopy types for simplicial complexes on $n$ vertices, but this exceeds the trivial upper bound of $2^{2^n}$.
\end{proof}
Based on this proposition, we end with the following conjecture which would imply that $\frac{\log_2 \log_2 h(n)}{n} \rightarrow 1$ as $n \rightarrow \infty$.
\begin{conjecture}
For $d \geq 2$ and $G$ an arbitrary finite abelian group
$$T_d(G) \leq (1 + o_d(1)) \frac{d}{\log(2) e} \sqrt[d]{\log|G|}$$
\end{conjecture}

\bibliography{ResearchBibliography}
\bibliographystyle{amsplain}
\end{document}